\documentclass[12pt]{amsart}

\usepackage[
text={440pt,575pt},
headheight=9pt,
centering
]{geometry}

\usepackage{amssymb}
\usepackage{bbm}
\usepackage{amsfonts}
\usepackage{amsthm}
\usepackage{graphicx}
\usepackage{color}
\usepackage{url}
\usepackage{microtype}
\usepackage{adjustbox}
\usepackage{tikz}
\usepackage{hyperref}
\usepackage{xspace}
\definecolor{darkblue}{RGB}{0,0,160}
\hypersetup{
colorlinks,%
citecolor=black,%
filecolor=black,%
linkcolor=darkblue,%
urlcolor=darkblue
}

\usepackage[utf8]{inputenc}

\newcommand{\excise}[1]{}

\newtheorem{thm}{Theorem}[section]
\newtheorem{lemma}[thm]{Lemma}

\newtheorem{prop}[thm]{Proposition}
\newtheorem{conj}[thm]{Conjecture}

\newtheorem{question}[thm]{Question}

\theoremstyle{definition}
\newtheorem{example}[thm]{Example}
\newtheorem{remark}[thm]{Remark}
\newtheorem{defn}[thm]{Definition}

\numberwithin{equation}{section}

\newcommand{\ring}[1]{\ensuremath{\mathbb{#1}}}

\newcommand\NN{\ring{N}}

\newcommand\RR{\ring{R}}

\newcommand\ZZ{\ring{Z}}

\newcommand\xx{{\mathbf x}}

\newcommand\ind{\mathbbm{1}}

\DeclareMathOperator{\supp}{supp} 
\DeclareMathOperator{\diag}{diag} 

\DeclareMathOperator{\conv}{conv} 

\newcommand{\yalmip}{\textsc{yalmip}\xspace}
\newcommand{\matlab}{\textsc{Matlab}\xspace}
\newcommand{\mosek}{\textsc{mosek}\xspace}

\begin{document}

\def\urladdrname{{\itshape Website}}
\title{Algebraic geometry of Poisson regression}

\author{Thomas Kahle}
\address{Otto-von-Guericke Universität\\ Magdeburg, Germany} 
\urladdr{\url{http://www.thomas-kahle.de}}

\author{Kai-Friederike Oelbermann}
\address{Otto-von-Guericke Universität\\ Magdeburg, Germany} 
 \urladdr{\url{http://www.kai-friederike.de}}

\author{Rainer Schwabe}
\address{Otto-von-Guericke Universität\\ Magdeburg, Germany} 
\email{rainer.schwabe@ovgu.de}

\thanks{TK is supported by the Research Focus Dynamical Systems of the
state Saxony-Anhalt.}

\date{\today}

\makeatletter
  \@namedef{subjclassname@2010}{\textup{2010} Mathematics Subject Classification}
\makeatother




\begin{abstract}
Designing experiments for generalized linear models is difficult
because optimal designs depend on unknown parameters.  Here we
investigate local optimality.  We propose to study for a given design
its region of optimality in parameter space.  Often these regions are
semi-algebraic and feature interesting symmetries.  We demonstrate
this with the Rasch Poisson counts model.  For any given interaction
order between the explanatory variables we give a characterization of
the regions of optimality of a special saturated design. This extends
known results from the case of no interaction.  We also give an
algebraic and geometric perspective on optimality of experimental
designs for the Rasch Poisson counts model using polyhedral and
spectrahedral geometry.
\end{abstract}

\maketitle
\setcounter{tocdepth}{1}
\tableofcontents

\section{Introduction}
\label{sec:introduction-1}

Generalized linear models are a mainstay of statistics, but optimal
experimental designs for them are hard to find, as they depend on the
unknown parameters of the model.  A common approach to this problem is
to study local optimality, that is, determine an optimal design per
fixed set of parameters.  In practice, this means that appropriate
parameters have to be guessed a priori, or fixed by other means.  Here
we take a global view on the situation.  Our goal is to partition
parameter space into \emph{regions of optimality}, such that in each
region the optimal design is (at least structurally) constant.  Our
key observation is that, by means of general equivalence theorems, the
regions of optimality are often \emph{semi-algebraic}, that is,
defined by polynomial inequalities.  This opens up the powerful
toolbox of real algebraic geometry to the analysis of optimality of
experimental designs.

We discuss the phenomenon on the Rasch Poisson counts model, a certain
generalized linear model that appears in Poisson regression, for
example in tests of mental speed in
psychometry~\cite{doebler2015processing}.  The parameterization of the
intensities of the Poisson distribution is akin to the toric models in
algebraic statistics.  The view from experimental design, however, is
new, and the resulting mathematical questions have not been considered
in algebraic statistics.  Our main result is a characterization of
optimality of a particular saturated design in
Theorem~\ref{t:cornerOpt}.  Beyond that, we also demonstrate how to
approach the problem from a geometric point of view.  In particular,
in Section~\ref{s:algGeo} we describe the problem of determining
regions of optimality in the language of mathematical optimization.
We are convinced that interesting mathematical structures can be found
when studying the polynomial inequalities that arise from the
different equivalence theorems in the theory of optimal experimental
design.

\subsection*{Acknowledgement}
TK is supported by the Research Focus Dynamical Systems (CDS) of the
state Saxony-Anhalt.  We thank Bernd Sturmfels for sharing his
geometric view on the problem, leading to some of the considerations
in Section~\ref{s:algGeo}.

\subsection*{Notation}
We switch freely between a binary vector $\xx = (x_i) \in\{0,1\}^k$
and a subset $A\subset \{1,\dots,k\}$.  If confusion can arise, the
subset corresponding to $\xx\in\{0,1\}^k$ is
written~$A(\xx)=\{i : x_i=1\}$, and conversely, the binary vector for
a given~$A$ is~$\xx(A)$ with components $x_i(A)=1$ if $i\in A$ and
$x_i(A) = 0$ otherwise.

\section{The Rasch Poisson counts model}
When testing mental speed, psychometrists often present series of
questions and count the number $Y$ of correctly solved items in a
fixed time.  One example of such a test is the Münster Mental Speed
Test~\cite{doebler2015processing}.  In such a setting it is natural to
model the \emph{response}~$Y$ as Poisson distributed with
parameter~$\lambda > 0$, often called the \emph{intensity}.  According
to the basic principle of statistical regression, the mean of the
response~$Y$ (which is just $\lambda$) is a deterministic function of
the factors of influence.  Rasch's idea was to make
$\lambda = \theta\sigma$ multiplicative in the ability $\theta$ of a
test person and the easiness $\sigma$ of the tasks.  Due to the
multiplicative structure, an absolute estimation of either ability or
easiness is only possible if the other quantity is fixed.  For the
mathematics, the distinction between $\theta$ and $\sigma$ is not
relevant, because we make another multiplicative ansatz for $\sigma$
below, and $\theta$ may well be subsumed there.

\emph{Rule based item generation} is a computer driven mechanism to
generate questions to present to the subjects.  One question's
easiness $\sigma(\xx)$ depends on a rule setting~$\xx$.  We think of
the \emph{rules} as discrete switches that can be on or off and that
influence the difficulty of the question.  In practice, we often
assume that each additional rule makes the task harder and thus
decreases the intensity.
Throughout the paper, the number of rules is fixed as $k\in\NN$.  The
possible experimental settings are thus the binary vectors
$\xx=(x_1,\ldots,x_k)\in \{0,1\}^k$ (but see our Notation section).

The natural choice for the influence of rule settings on the intensity
$\lambda$ is exponential:
\begin{equation}\label{e:fAnsatz}
\lambda (\xx) = \theta\sigma (\xx) = \exp (f(\xx)^T\beta)
\end{equation}
for a vector of \emph{regression functions} $f : \{0,1\}^k \to \RR^p$,
and a vector of parameters~$\beta\in\RR^p$.  A concrete \emph{model}
is specified by means of the integers $k,p$, and the regression
functions~$f$.

\begin{defn}\label{d:interaction}
The \emph{interaction model of order $d$} is specified by the
regression function
\begin{equation}\label{eq:regression}
f_d (\xx) = (\text{all squarefree monomials of degree at most $d$ in $x_1,\dots,x_k$})
\end{equation}
\end{defn}

\begin{remark}\label{r:squarefree}
If our rule settings $\xx$ were not binary, then in
Definition~\ref{d:interaction} there would be a difference between
using all monomials and all squarefree monomials.  For binary $\xx$
there is none since $x_i^2=x_i$ for all~$i$.
\end{remark}

\begin{example}\label{e:regression}
The most interesting model from a practical perspective is the
\emph{independence model} which arises for $d=1$.  In this case
$f(\xx) = (1,x_1,\dots,x_k)$ and $p=1+k$.  The \emph{pairwise
interaction model} arises for $d=2$, where
$f(\xx) = (1,x_1,\dots,x_k,x_1x_2,\dots,x_{k-1}x_k)$ and
$p = 1+k+\binom{k}{2}$.  Somewhat confusingly this second situation is
sometimes called \emph{first order interaction}.
\end{example}

Definitions \eqref{e:fAnsatz} and \eqref{eq:regression} lead to a
product structure for the intensity~$\lambda(\xx)$ as follows.  Let
$d\ge 1$.  There is a parameter $\beta_A$ for each
$A\subset \{1,\dots,k\}$ with $|A|\le d$.  Then
\begin{equation} \label{eq:productStructure} \lambda(\xx,\beta) =
\prod_{\substack{A\subset A(\xx)\\|A|\le d}} e^{\beta_A}.
\end{equation}
Hence, the more rules are applied, the more terms $e^{\beta_A}$ enter
the product~\eqref{eq:productStructure}.  In the $d=1$ case, there is
one term $e^{\beta_{\{i\}}}$ for each $i\in\{1,\dots,k\}$ and one
global term~$\beta_\emptyset$.  The intensity is then proportional to
the product over those terms for which the corresponding rule is
active and there is no interaction among the rules.  For higher
interaction order $d$, if, for example, rules $1,2$ are active, the
corresponding factor is
$e^{\beta_{\{1\}}}e^{\beta_{\{2\}}}e^{\beta_{\{1,2\}}}$, etc.  If all
singleton parameters $\beta_{\{i\}}$ have the same sign, then having a
parameter $\beta_A, |A|\ge 2$ with the same sign is sometimes called
\emph{synergetic interaction}, while $\beta_A$ with a different sign
is called \emph{antagonistic interaction}.

The case $d=1$ is particularly well-behaved (and very relevant for the
practitioners).  Graßhoff, Holling, and Schwabe have investigated this
case in depth in
\cite{grasshoff2013optimal,grasshoff2014optimal,grasshoff2015stochastic}.
In Section~\ref{sec:saturated-designs} we generalize some of their
results to the general interaction case.

\begin{remark}\label{r:simplicialComplex}
In \eqref{eq:regression} we chose all squarefree monomials of bounded
degree.  Therefore, if there is a parameter $\beta_A$ for some set $A$
of rules, then there also are parameters $\beta_B$ for all subsets $B$
of~$A$.  In the language of combinatorics, the indices of the
parameters form a \emph{simplicial complex}, and one could conversely
define a model for each simplicial complex, by letting the regression
function consist of squarefree monomials corresponding to the faces of
the complex.  This puts our parametrizations of possible intensities
$\lambda$ in the context of \emph{hierarchical log-linear
models}~\cite[Section~1.2]{drton09:_lectur_algeb_statis}, certain
hierarchically structured exponential families that also arise in the
theory of information processing systems~\cite{KahOlbJosAy08}.
\end{remark}

\begin{remark}\label{r:MarkovBases}
In \eqref{e:fAnsatz}, not all vectors $\lambda \in \RR^{2^k}$ have
corresponding parameters~$\beta$.  Obviously, $\lambda$ needs to have
positive entries, but there are further restrictions.  For the
simplest example, in the case $k=2, d=1$, there are four possible rule
settings $\{(00), (01), (10), (11)\}$.  Independent of the parameters
$\beta$, it holds that
\[
\lambda(00)\lambda(11) - \lambda(10)\lambda(01) = 0,
\]
since both terms equal $e^{2\beta_\emptyset}e^{\beta_1}e^{\beta_2}$.
As a function $\RR^4 \to \RR$, this $2\times 2$ determinant vanishes
identically on the image of the parametrization and it can be seen
that this vanishing characterizes points in the image.  For any $k$
and $d$, there is a finite set of binomials (that is, polynomials with
only two monomials) in $\lambda$ that characterizes the image of the
parameterizations.  In commutative algebra these are known as the
generators of certain toric
ideals~\cite[Chapter~4]{sturmfels96:_gr_obner_bases_and_convex_polyt},
while in algebraic statistics they are called Markov
bases~\cite{diaconissturmfels98}.  In principle, after fixing $k$
and~$d$, all binomials can be computed with the help of computer
algebra (the fastest software is \text{4ti2}~\cite{4ti2}), but this is
hard already for $d=2$ and $k>7$.  Many special cases have been dealt
with in algebraic statistics, though.  See \cite{aoki2012markov} and
references therein.
\end{remark}

\subsection{Optimal experimental design}
The estimation problem is to determine the values of the parameters
$\beta$ given observations $(Y^{(i)},\xx^{(i)})$, $i=1,\dots, N$ which
are pairs of experimental settings $\xx^{(i)}$ and
responses~$Y^{(i)}$.  As $Y$ is a random variable, the estimator is a
random variable too.  In practice, when designing an experiment to
determine~$\beta$, we can choose which settings $\xx^{(i)}$ to
present.  This choice should be made so that the result of the
experiment is most informative about~$\beta$.  Doing so, we may also
choose to test a particular setting $\xx$ multiple times.  This
quickly leads to an idea of Kiefer: An \emph{approximate design} is a
vector $(w_\xx)_{\xx\in\{0,1\}^k} \in [0,1]^{2^k}$ of non-negative
weights with $\sum_\xx w_\xx = 1$.  In the following we only work with
approximate designs as our choices of experimental settings.

How is the quality of a design to be measured?  Quite generally, one
uses the Fisher information matrix, defined as
\begin{equation}\label{eq:FischerInformation}
M (w,\beta) =\sum_{\xx \in \{0,1\}^k} w_\xx \lambda(\xx, \beta)
f(\xx) f(\xx)^T.
\end{equation}
This choice can be motivated by large sample asymptotics:
asymptotically the maximum-likelihood-estimator of the parameters is
normal and its standardized covariance matrix is the inverse of the
Fisher information.  An \emph{optimality criterion} is any function
that produces a real number from the Fisher information.  Here we
choose the popular $D$-optimality criterion which uses the
determinant.  The design problem for the Poisson counts model is to
determine descriptions of the regions in $\beta$-space where certain
designs are optimal.  Given a particular design, however, there may be
no parameters $\beta$ for which this design is optimal.

\begin{remark}\label{r:betaempty}
When the global parameter $\beta_\emptyset$ changes, the determinant
of $M(w,\beta)$ is globally scaled.  For all question regarding
optimal design we may therefore assume $\beta_\emptyset = 0$.
\end{remark}

\begin{example}\label{e:betaZero}
If $\beta_A = 0$ for all $A$, one can check that the situation reduces
to that of a linear model.  A $D$-optimal experimental design is given
by the uniform weight vector $w_\xx = \frac{1}{2^k}$, for
all~$\xx\in\{0,1\}^k$.
\end{example}

\subsection{Symmetry}
\label{s:symmetry}
The regions of optimality show a high degree of symmetry.  We use only
basic facts about symmetric designs.  Corresponding statements can be
made in more general settings~\cite{radloff15:_invar}.  Let $G$ be a
finite group acting on the set of design points~$\{0,1\}^k$.  Two
natural symmetries result from $G = S_k$, the symmetric group
permuting rules, and $G=\ZZ_2^k$ whose elements exchange the roles of
$0$ and~$1$ for some rules.  The action $\circ$ of $G$ on approximate
designs is defined by $(g\circ w)_\xx = w_{g\circ \xx}$.  A crucial
assumption for the exploitation of symmetry in design theory is that
the action of $G$ induces a linear action on regression functions,
that is, for each $g\in G$ there is a matrix $Q_g$ such that
$f(g\circ\xx) = Q_gf(\xx)$.  It is not difficult to assert this
assumption in our case.  From this one can define a corresponding
action (also denoted $\circ$) on parameter space via the requirement
$f(g\circ \xx)^T(g\circ\beta) = f(\xx)^T\beta$ that the response be
invariant.  It is obvious that $g\circ \beta = Q_g^{-T}\beta$ is a
possible choice.  By linearity, and since the intensity
$\lambda(\xx,\beta)$ only depends on the response $f(\xx)^T\beta$,
information matrices transform as
\[
M(g\circ w, g\circ \beta) = Q_g M(w,\beta) Q_g^T.
\]
Since $G$ is finite, $Q_g$ is unimodular and the determinant is
unchanged.  This proves that any optimal design $w$ for parameters
$\beta$ yields the optimal design $g\circ w$ for parameters
$g\circ \beta$: if a better value was possible in the optimization
problem for parameters $g\circ \beta$, then a better value of the
determinant could also be achieved in the problem for~$\beta$.  In
total we have the following proposition.
\begin{prop}\label{p:symmetric}
The regions of optimality are symmetric in the sense that if $w$ is a
$D$-optimal approximate design for parameters $\beta$, then $g\circ w$
is a $D$-optimal approximate design for parameters $g\circ \beta$.
\end{prop}

\begin{example}\label{e:signChange}
If $d=1$, and $G=\ZZ_2^k$ consists of $0/1$ exchanges, then it is easy
to check that the matrices $Q_g$ correspond to sign changes on the
parameters~$\beta$.  In particular, the regions of optimality are
point symmetric around the origin.
\end{example}

Another way to study the symmetry in this optimization problem is to
compute the determinant of the information matrix explicitly.  For
example, if $d=1$, exchanging $\beta_i$ by $-\beta_i$ replaces
$\lambda_i$ by $1/\lambda_i$ so that homogeneity of the determinant
can be exploited to see the symmetry.  For $d=1,k=2$, the determinant
is equal to the elementary symmetric polynomial of degree three in the
products~$w_{\xx}\lambda(\xx)$.  For $d=1$ and higher values $k$, the
determinant is not an elementary symmetric polynomial (it misses
monomials) but it still has a nice combinatorial description.  It is
an interesting exercise to work out the relation between the
determinant and the matrices $Q_g$ from above also for~$d>1$.

\section{Semi-algebraic regions of optimality for saturated designs}
\label{sec:saturated-designs}
The number of parameters of our interaction model of order $d$ equals
$p=\sum_{i=0}^d \binom{k}{i}$.  The Fisher information matrix
in~\eqref{eq:FischerInformation} is of format $p\times p$.
Carath\'eodory's theorem implies that every Fisher information matrix
is realized by a design $w$ which has at most $\frac{1}{2} p (p-1) +1$
support points.  Both the support points and the corresponding weights
are in general not unique, but in certain situations the optimal
experimental design is quite rigid.  A design is \emph{saturated} if
it is supported on exactly $p$ points.  It is clear
from~\eqref{eq:FischerInformation} that this is the minimal number of
points, since a convex combination of less than $p$ rank one matrices
has rank at most~$p-1$.  For saturated designs it is well-known that
$D$-optimal weights are uniform, that is, all weights $w_\xx$,
$\xx\in\supp (w_\xx)$ are equal to $1/p$ (see
\cite[Corollary~8.12]{pukelsheim1993optimal}).  Hence, optimization in
the class of saturated designs reduces to the choice of $p$
experimental settings~$\xx$ appearing in the support of~$w$.  We now
define a special design whose optimality we can characterize.  Its
support points correspond exactly to the terms in the regression
function.

\begin{defn}\label{d:corner}
The \emph{corner design $w_{k,d}^*$} is the saturated design with
equal weights $w_{\xx}=1/p$ for all $\xx \in \{0,1\}^k$ with
$|\xx|_1 \le d$.
\end{defn}

\begin{example}
For $k=3$ rules and interaction order $d=2$ the regression function is
$f(x_1,x_2,x_3)=(1,x_1,x_2,x_3,x_1x_2, x_1x_3, x_2x_3)$ and there are
$p=7$ parameters.
The corner design has weight $1/7$ on the seven binary 3-vectors not
equal to $(1,1,1)$:
\[
w_{3,2}^*: \quad w_{(0,0,0)}=w_{(1,0,0)}=w_{(0,1,0)}=w_{(0,0,1)}=
w_{(1,1,0)}=w_{(1,0,1)}=w_{(0,1,1)}=1/7.
\]
\end{example}

We introduce the shorthand notation $\mu_A := e^{\beta_A}$ and
$\mu_{i} = \mu_{\{i\}}$, $\mu_{ij} = \mu_{\{i,j\}}$, etc.  The region
of optimality of the corner design is described in the following
theorem.  Its proof is a translation of the inequalities in the
Kiefer-Wolfowitz equivalence
theorem~\cite[Section~9.4]{pukelsheim1993optimal} and follows later in
Section~\ref{sec:proof-theorem} after some discussion of consequences
and relations to existing work.
\begin{thm}{\label{t:cornerOpt}}
The corner design $w^*_{k,d}$ is optimal if and only if for all
$C\subseteq \{1,\dots,k\}$ with $|C|>d$
\begin{align}\label{eq:cornerOpt}
  \sum_{\substack{B\subset C \\ |B| \le d}} \binom{|C| - |B| -
  1}{d - |B|}^2 
  \prod_{\substack{A\subset C, |A| \le d \\A\neq B}}
  \mu_A \leq 1.
\end{align}
\end{thm}
The inequalities in Theorem~\ref{t:cornerOpt} can always be satisfied.
Indeed by making parameters $\beta_A$ sufficiently negative, the left
hand side of \eqref{eq:cornerOpt} can be made as small as desired.
This has the interpretation that, if the rules make the problem hard
enough, not testing particularly hard settings becomes eventually
optimal.  Stated geometrically: The region of optimality of the corner
design is non-empty, independent of the interaction order and the
number of rules.

In the case $d=1$, Graßhoff et al. have shown that almost all of the
inequalities in Theorem~\ref{t:cornerOpt} are redundant.
Specifically, \cite[Theorem~1]{grasshoff2014optimal} shows that if
\[
\mu_{ij}+ \mu_i +\mu_j \leq 1
\]
for all pairs of $1\leq i < j \leq k$ then the corner design
$w_{k,1}^*$ is $D$-optimal.  In \eqref{eq:cornerOpt}, these
inqualities correspond to $|C|=2$.  The remaining inequalities are all
redundant and can be omitted.  This is not the case if $d>1$ as
illustrated by the following example.

\begin{example}\label{e:notRedundant}
Let $d=2$.  For $k=4$, fixing $\mu_\emptyset = 1$ with
Remark~\ref{r:betaempty}, the Rasch Poisson counts model has 10
remaining parameters $\mu_1,\dots,\mu_4,\mu_{12},\dots,\mu_{34}$.
Theorem~\ref{t:cornerOpt} stipulates five inequalities that
characterize optimality of the corner design.  Four of the
inequalities correspond to the four subsets of size three.  For
example, the inequality for $C = \{1,2,3\}$ has terms of degrees six
and five:
\begin{equation}\label{eq:exampleIneq1}
\mu_1\mu_2\mu_3\mu_{12}\mu_{13}\mu_{23} +
\mu_2\mu_3\mu_{12}\mu_{13}\mu_{23} + 
\dots + 
\mu_1\mu_2\mu_3\mu_{12}\mu_{13} \le 1.
\end{equation}
The inequality corresponding to $C = \{1,2,3,4\}$ has non-trivial
binomial coefficients and terms of degrees ten and nine:
\begin{equation}\label{eq:exampleIneq2}
9 \prod_{|A|\leq 2} \mu_A + 4 \sum_{|B|=1} \prod_{\substack{|A| \leq
2\\A\neq B}} \mu_A + \sum_{|B|=2} \prod_{\substack{|A| \leq
2\\A\neq B}} \mu_A \le 1.
\end{equation}
To confirm that the final inequality is not redundant, we are
searching for a point that satisfies all four inequalities
in~\eqref{eq:exampleIneq1}, but not that in~\eqref{eq:exampleIneq2}.
To reduce dimension, we restrict to parameter values invariant under
the symmetric group $S_k$ permuting rules.  For singletons $i$, let
$\mu_i=s$ and for pairs $\{i,j\}$, $i\neq j$, let $\mu_{ij}=t$.  In
this two-dimensional set of parameter values, the inequalities take
the form
\begin{equation*}
s^3t^3 + 3s^2t^3 + 3s^3t^2 \le 1,\qquad\quad
9s^4t^6 + 16s^3t^6 + 6s^4t^5 \le 1.
\end{equation*}
It is easy to verify that $s=5/9, t=4/5$ satisfies the first
inequality, but violates the second.  Figure~\ref{fig:stExample} is a
plot of the resulting inequalities for $k=10$.  The region of
optimality consists of all points that lie below any of the curves.
\begin{figure}[htb]
\centering
\newcommand{\shiftleft}[2]{\makebox[0pt][r]{\makebox[#1][l]{#2}}}
\includegraphics[width=.8\linewidth]{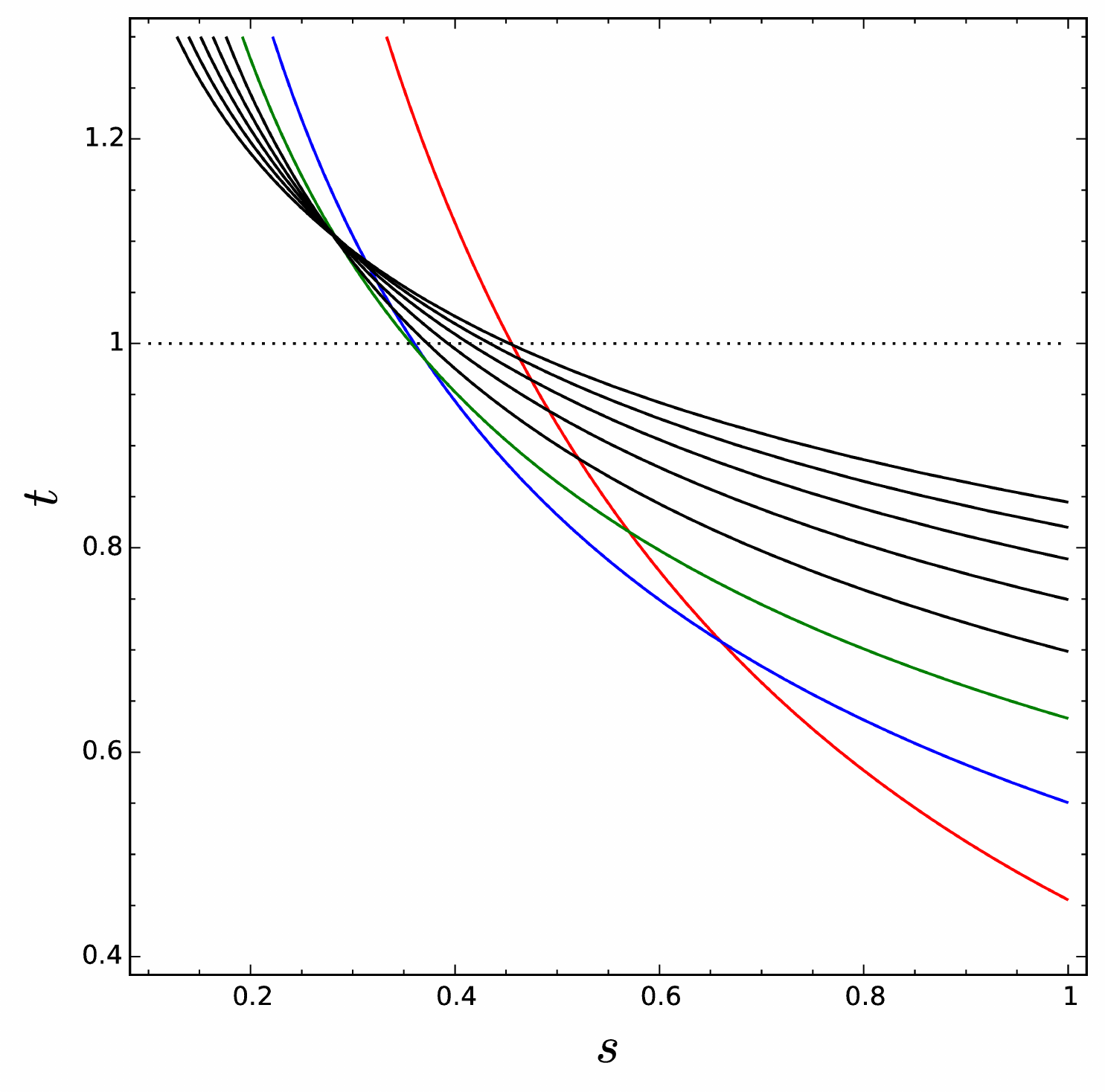}
\raisebox{1.7cm}{
  \shiftleft{.71\linewidth}{
    \includegraphics[height=4cm]{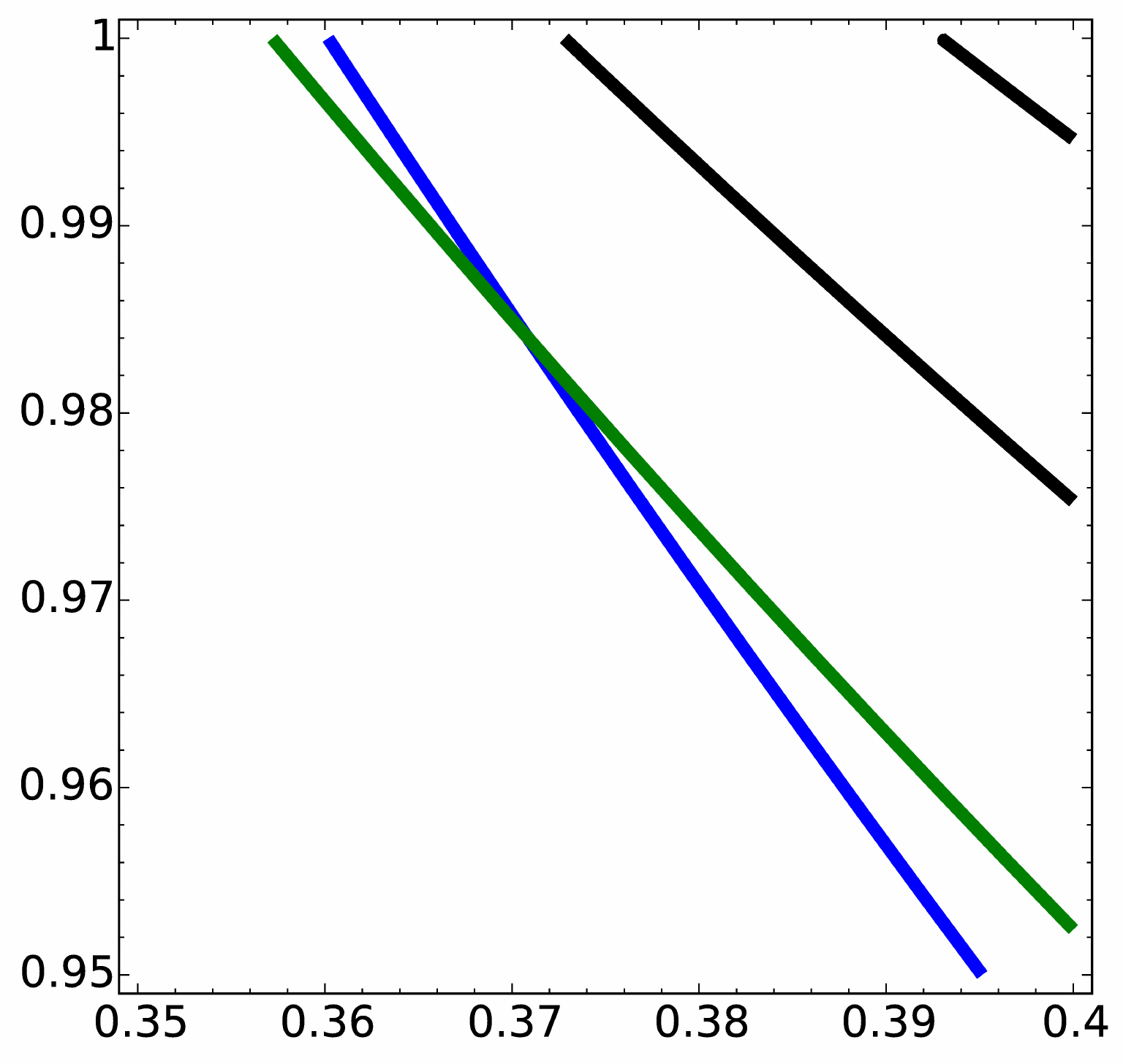}}}
    \caption{\label{fig:stExample} The curves in this plot consist of
    the points for which the inequalities \eqref{eq:cornerOpt} are
    attained.  The region of optimality of the corner design in
    Example~\ref{e:notRedundant} is the region below any of the
    curves.  On the right hand side of the picture, $|C|=3$ is the
    lowermost curve, then $|C|=4$, and so on.  Consequently, the red,
    blue, green curves correspond to $|C|=3,4,5$, respectively.  In
    the region where $s,t \le 1$ (that is, $\beta_A \le 0$),
    inequalities corresponding to $|C| > 5$ are redundant and plotted
    in black.  The inset shows a tiny region where the $|C| = 5$
    inequality is necessary.  The region above the dotted line
    corresponds to antagonistic interaction: Two rules being active at
    the same time make the problem easier.  In this case also
    inequalities for $|C| > 5$ are tight.}
\end{figure}
\end{example}
For fixed $d$, as $k$ grows larger, more inequalities arise from
Theorem~\ref{t:cornerOpt}.  We conjecture that when $\beta_A<0$ for
all $A$, $|A|\le d$, as $k$ grows, a finite number of them suffices to
characterize the region of optimality of~$w^*_{k,d}$.
\begin{conj}
Fix $d$ and assume all parameters have negative values: $\beta_A < 0$,
$|A|\le d$.  There exists a constant $c(d)$ such that in
Theorem~\ref{t:cornerOpt} the inequalities corresponding to $C$ with
$|C|>c(d)$ are redundant given the remaining ones.  In particular,
$c(2) = 5$.
\end{conj}

\begin{remark}\label{r:inParameters}
The inequalities \eqref{eq:cornerOpt} are restrictions on the
parameters.  For $d=1$ it happens that they can be rewritten as
inequalities in the intensities $\lambda (\xx,\beta)$, but in general
this is not the case.  In principle a semi-algebraic description in
parameter space can computed.  With $\phi$ the parametrization mapping
coordinates $\mu$ to intensities $\lambda$, consider the set
$\left\{ (\mu,\lambda) : \lambda=\phi(\mu), \mu \text{ satisfies
  \eqref{eq:cornerOpt}}\right\}$.
According to the Tarski--Seidenberg theorem the projection of this
semi-algebraic set to the $\lambda$ coordinates is again
semi-algebraic.  Actual computation, however, relies on quantifier
elimination.  Therefore even the best algorithms are for now unable to
solve simple examples.  See \cite{basu2006algorithms} for the theory
of such computations.
\end{remark}

We finish the discussion with a question regarding other saturated
designs.
\begin{question}\label{q:cornerOnly}
When $\beta_A < 0$, for all $A$, $|A|\le d$, is the corner design the
only saturated design that admits $D$-optimal parameter values?
\end{question}
The Kiefer-Wolfowitz theorem gives a system of inequalities for any
saturated design and this system characterizes parameter values for
optimality.  In the case $d=1, k=3$ Graßhoff et al. have shown that,
up to fractional factorial designs at $\beta=0$, only the corner
design yields a feasible system~\cite{grasshoff2015stochastic}.  We
have used numerical moment relaxations and semi-definite programming
to numerically confirm the case $d=1,k=4$.  Everything beyond this is
out of reach at the moment.

\subsection{Proof of Theorem~\ref{t:cornerOpt}}
\label{sec:proof-theorem}
The Kiefer-Wolfowitz theorem characterizes regions of optimality of a
fixed saturated design $w$ by means of inequalities in
parameters~$\mu_A$ (or equivalently $\beta_A$).  We apply it to the
corner design and make these inequalities explicit.  To do so, a
0/1-matrix needs to be inverted.
\begin{defn}
For fixed $k,d$, the \emph{model matrix} $F_{k,d}$ is the matrix whose
rows are the regression
vectors~$\{f(\xx) : \xx\in\supp (w^*_{k,d})\}$.
\end{defn}
The rows and columns of $F_{k,d}$ may be indexed by subsets of
$A\subseteq \{1,\dots,k\}$ with $|A|\leq d$ so that $F_{k,d}$ is lower
triangular.  We ommit the subscript indices if $k,d$ are fixed or
clear from the context.

\begin{example}
For $k=3$ and $d=2$ the model matrix is
\[
F_{3,2}=
\bordermatrix{
& 1& x_1 & x_2 &x_3 & x_1x_2 &x_1x_3 & x_2x_3 \cr
1 &1& 0 & 0 &0 & 0&0&0 \cr
x_1 &1& 1 & 0 &0 & 0&0&0 \cr
x_2 &1& 0 & 1 &0 & 0&0&0 \cr
x_3 &1& 0 & 0 &1 & 0&0&0 \cr
x_1x_2 &1& 1 & 1 &0 & 1&0&0 \cr
x_1x_3 &1& 1 & 0 &1 & 0&1&0 \cr
x_2x_3 &1& 0 & 1 &1 & 0&0&1 \cr }.
\]
\end{example}

In the general setup of $k$ rules and interaction order $d$ the
entries $F_{A,B}$ of $F$ are
\[
F_{A,B} =
\begin{cases}
1 & \text{ if } B\subseteq A \\
0 & \text{ otherwise,}
\end{cases} \qquad \text{where }\, A,B \subset \{1,\dots,k\}, |A|\le d, |B|\le d.
\]

\begin{lemma}\label{l:invertF}
The inverse of $F$ has entries
\[
F^{-1}_{A,B} =
\begin{cases}
(-1)^{|A|-|B|} & \text{ if } B \subseteq A \\
0 & \text{ otherwise.} 
\end{cases}
\]
\end{lemma}
\begin{proof}
Matrix multiplication yields
\begin{equation*}
  (F\cdot F^{-1})_{A,B}
  = \sum_{\substack{C\subset \{1,\dots,k\}\\|C| \le d}} F_{A,C}F^{-1}_{C,B} =
  \sum_{\substack{C\subset\{1,\dots,k\}\\|C|\le d}} (-1)^{|A|-|B|}
  \ind_{B\subset C}\ind_{C\subset A}.
\end{equation*}
If $A = B$, then there is only one summand, so that
$(F\cdot F^{-1})_{A,A} = 1$.  If not, then since both matrices are
lower triangular, consider the case $B \subsetneq A$ and let
$a\in A\setminus B$.  There is a bijection which identifies a set $C$
not containing $a$ with $C \cup \{a\}$.  This bijection matches
summands with opposite signs and consequently
$(F\cdot F^{-1})_{A,B} = 0$.
\end{proof}

If $|\xx| \le d$, then there is a row with index $B$ in $F$ that may
be identified with~$\xx$ via $F(\xx)_B=(1,\xx, \ldots)$.  For this
$\xx$ we have
\[
\left(F^{-T}f(\xx) \right)_A = (e_{\xx})_A := \begin{cases}
1 & A = B \\
0 & \text{otherwise}.
\end{cases}
\]

This is a special case of the following lemma.
\begin{lemma}\label{l:multiply}
Let $\xx \in \{0,1\}^k$ then
\[
(F^{-T}f(\xx))_A =
\begin{cases}
(-1)^{d-|A|}\binom{|A(\xx)|-|A|-1}{d-|A|} & \text{if } A \subset A(\xx) \\
0 & \text{otherwise}.
\end{cases}
\]
\end{lemma}

\begin{proof}
We compute
\[
(F^{-T} f(\xx))_A = 
\sum_{|B|\le d} (-1)^{|A|+|B|} \ind_{A\subset B}\ind_{B\subset A(\xx)}.
\]
If $A\not\subset A(\xx)$ then all summands are zero.  Therefore we can
relabel the summands by sets $B'$ disjoint from $A$ such that
$B = A \cup B'$.  This yields
\[
(F^{-T} f(\xx))_A = 
  \sum_{\substack{B'\subset A(\xx)\setminus A\\|B'|\leq d-|A|}}
  (-1)^{|A| + |B'| + |A|}.
\]
The result now follows from the following known formula (which is also
easy to prove by induction)
\[
\sum_{k=0}^K (-1)^k \binom{n}{k} = (-1)^K \binom{n-1}{K}. \qedhere
\]
\end{proof}

\begin{proof}[Proof of Theorem~\ref{t:cornerOpt}]
 
By the Kiefer-Wolfowitz theorem, a saturated design is optimal if and
only if the following inequality holds for all settings
$\xx \in \{0,1\}^k$
\[
\lambda(\xx) (F^{-T}f(\xx))^T \Psi^{-1} (F^{-T}f(\xx)) \le 1,
\]
where $\Psi=\diag(1, (\mu_A)_{|A|\leq d})$.
The inequalities corresponding to $\xx$ with $|\xx|\le d$ are
automatically satisfied with equality:
\begin{equation*}
\lambda(\xx) (F^{-T}f(\xx))^T \Psi^{-1} (F^{-T}f(\xx))
=\lambda(\xx) e_{\xx}^T \Psi^{-1} e_{\xx}
=\prod_{\substack{A \subset \xx,\\ |A|\leq d}} \mu_A 
  \prod_{\substack{A \subset \xx,\\ |A|\leq d}} \mu_A^{-1}
= 1.
\end{equation*}
When $|\xx|\ge d+1$, using Lemma~\ref{l:multiply}, we get the
inequalities
\begin{equation*}
\sum_{\substack{B\subset A(\xx) \\ |B| \le d}} \binom{|A(\xx)|-|B| -
1}{d - |B|}^2 \prod_{\substack{A\subset A(\xx), |A| \le d \\A\neq B}}
\mu_A \leq 1. \qedhere
\end{equation*}
\end{proof}

\begin{remark}
In the proof of Theorem~\ref{t:cornerOpt}, when $|\xx| = d+1$, by
Lemma~\ref{l:multiply}, the entries $(F^{-T}f(\xx))_A$ are zero if
$A\not\subset A(\xx)$ and equal to $\pm 1$ if $A\subset A(\xx)$, since
in this case the binomial coefficient is~$\binom{d-|A|}{d-|A|}$.  We
then find inequalities of the form
\begin{equation*}
\sum_{\substack{B\subset A(\xx)\\ |B| \le d}} \prod_{\substack{A\subset
A(\xx), |A| \le d \\A\neq B}} \mu_A \leq 1.
\end{equation*}
\end{remark}

\section{A geometric perspective on \texorpdfstring{$D$}{D}-optimal designs}
\label{s:algGeo}
For each $\xx$, the matrix $f(\xx)f(\xx)^T$ is a positive-semidefinite
rank one matrix with entries zero and one.  They are the vertices of
the optimization domain which turns out to be a polytope:
\begin{defn}\label{d:polytope} The
\emph{information matrix polytope} is
\[
P(\beta) = \conv \left\{\lambda(\xx,\beta) f(\xx)f(\xx)^T :
\xx\in\{0,1\}^k \right\}.
\]
\end{defn}
All points of which the convex hull is taken are also vertices
of~$P(\beta)$, since any affine combination of them has rank at least
two.  Each point in $P(\beta)$ is an information matrix $M(w,\beta)$
for some approximate design~$w$.  In the case $\beta=0$ (which implies
$\lambda(\xx) = 1$ for all~$\xx$), the arising polytopes are well-known
in the combinatorial optimization literature.
\begin{example}\label{e:correlationMarginal}
When $d=1$ and $\beta=0$, $P(\beta)$ is the correlation polytope.  To
make this obvious, one needs to omit the constant entry $1$ from the
beginning of the regression function~$f$.  The correlation polytope is
well-known in combinatorial optimization and its complexity provides
lower complexity bounds there~\cite{kaibel2013short}.  It is affinely
equivalent to the even better known cut polytope via the covariance
mapping~\cite[Chapter~5]{dezalaurent97}.  For higher $d$, and
$\beta=0$, the polytope $P(\beta)$ is called an \emph{inclusion
polytope} in~\cite[Section~2.4.1]{kahle10:_bound_statis_model}.  It is
affinely equivalent (via a generalization of the covariance mapping)
to the marginal polytope of a corresponding hierararchical model.
\end{example}

The problem of determining an optimal experimental design has two
steps
\begin{enumerate}
\item\label{it:optimize} Determine an optimal information matrix~$M^*$.
\item\label{it:decompose} Determine weights $w$ that write the optimal
matrix $M^*$ as a convex combination of
vertices~$\lambda(\xx,\beta) f(\xx)f(\xx)^T$ of the information matrix
polytope.
\end{enumerate}
The possible solutions to the second problem are dealt with using
convex geometry.  In particular Carath\'eodory's theorem applies and
gives bounds for support sizes of weight vectors~$w$.

In the case of $D$-optimality, the optimization problem in
step~\ref{it:optimize} is to maximize the determinant over~$P$.  The
determinant vanishes at the vertices of $P$, and since it is a
log-concave function, a unique maximum with positive value is attained
in the interior, as soon as there are full rank matrices in the
interior.  All matrices in the information matrix polytope are
positive semidefinite.  This motivates the \emph{linear matrix
inequality (LMI) relaxation of $P(\beta)$}.  For this, the
optimization domain $P(\beta)$ is replaced by the \emph{spectrahedron}
arising as the intersection of the cone of positive semidefinite
matrices with the affine space spanned by~$P(\beta)$.

Maximization of the determinant over a spectrahedron is a well-known
convex optimization problem~\cite{vandenberghe1996semidefinite}.  The
unique point where the determinant is maximal is known as the
\emph{analytic center} of the semidefinite program.  If the analytic
center of the linear matrix inequality lies inside $P(\beta)$, then it
gives the optimal experimental design.  It is therefore an interesting
problem to give a fully geometric description of the case that the
analytic center lies outside of~$P$.

\begin{question}
For fixed $k,d$, as a function of $\beta$, what is the difference
between $P(\beta)$ and its LMI relaxation?  Through which faces can
the analytic center leave $P(\beta)$ when $\beta$ changes?
\end{question}

\begin{example}\label{e:lmi}
Let $k=2$ and $d=1$.  Setting again $\beta_\emptyset = 0$, the two
parameters of the Rasch Poisson counts model are
$\lambda_i = e^{\beta_i}$, $i=1,2$.  By symmetry considerations from
Section~\ref{s:symmetry} we restrict ourselves to $\beta_i \le 0$,
which corresponds to $\lambda_i \in (0,1]$.  The information matrix
polytope is
\[
P = \conv \left\{
  \begin{pmatrix}
  1 & 0 & 0 \\
  0 & 0 & 0 \\
  0 & 0 & 0 
  \end{pmatrix},
  \begin{pmatrix}
  \lambda_1 & \lambda_1 & 0 \\
  \lambda_1 & \lambda_1 & 0 \\
  0 & 0 & 0 
  \end{pmatrix},
  \begin{pmatrix}
  \lambda_2 & 0 & \lambda_2 \\
  0 & 0 & 0 \\
  \lambda_2 & 0 & \lambda_2 
  \end{pmatrix},
  \begin{pmatrix}
  \lambda_1\lambda_2 & \lambda_1\lambda_2 & \lambda_1\lambda_2 \\
  \lambda_1\lambda_2 & \lambda_1\lambda_2 & \lambda_1\lambda_2 \\
  \lambda_1\lambda_2 & \lambda_1\lambda_2 & \lambda_1\lambda_2 
  \end{pmatrix}
\right\}.
\]
Independent of the values $\lambda_1,\lambda_2$, the polytope $P$ is a
3-dimensional simplex.  Its LMI relaxation is the intersection of the
cone of $(3\times 3)$ positive-semidefinite matrices with the affine
space spanned by~$P$.  This yields the following linear matrix
inequality ($\succeq 0$ means positive-semidefinite), using the first
vertex as the base point and variables~$x,y,z$: \small
\[
\left\{
(x,y,z) : 
      \begin{pmatrix}
      1 + x (\lambda_1-1) + y (\lambda_2-1) + z(\lambda_1\lambda_2 -1)
      & \lambda_1x + \lambda_1\lambda_2z  
      & \lambda_2y + \lambda_1\lambda_2z \\
      \lambda_1x + \lambda_1\lambda_2 z & \lambda_1x + \lambda_1\lambda_2z &
      \lambda_1\lambda_2z \\
      \lambda_2y + \lambda_1\lambda_2 z & \lambda_1\lambda_2z & \lambda_2y +
      \lambda_1\lambda_2z
      \end{pmatrix} 
      \succeq 0
\right\}.
\]
\normalsize Figure~\ref{f:surferSpectra} contains plots of the
resulting spectrahedra ``along the diagonal''
$\lambda := \lambda_1 = \lambda_2$.
\begin{figure}[htp]
\centering
\includegraphics[width=.32\textwidth]{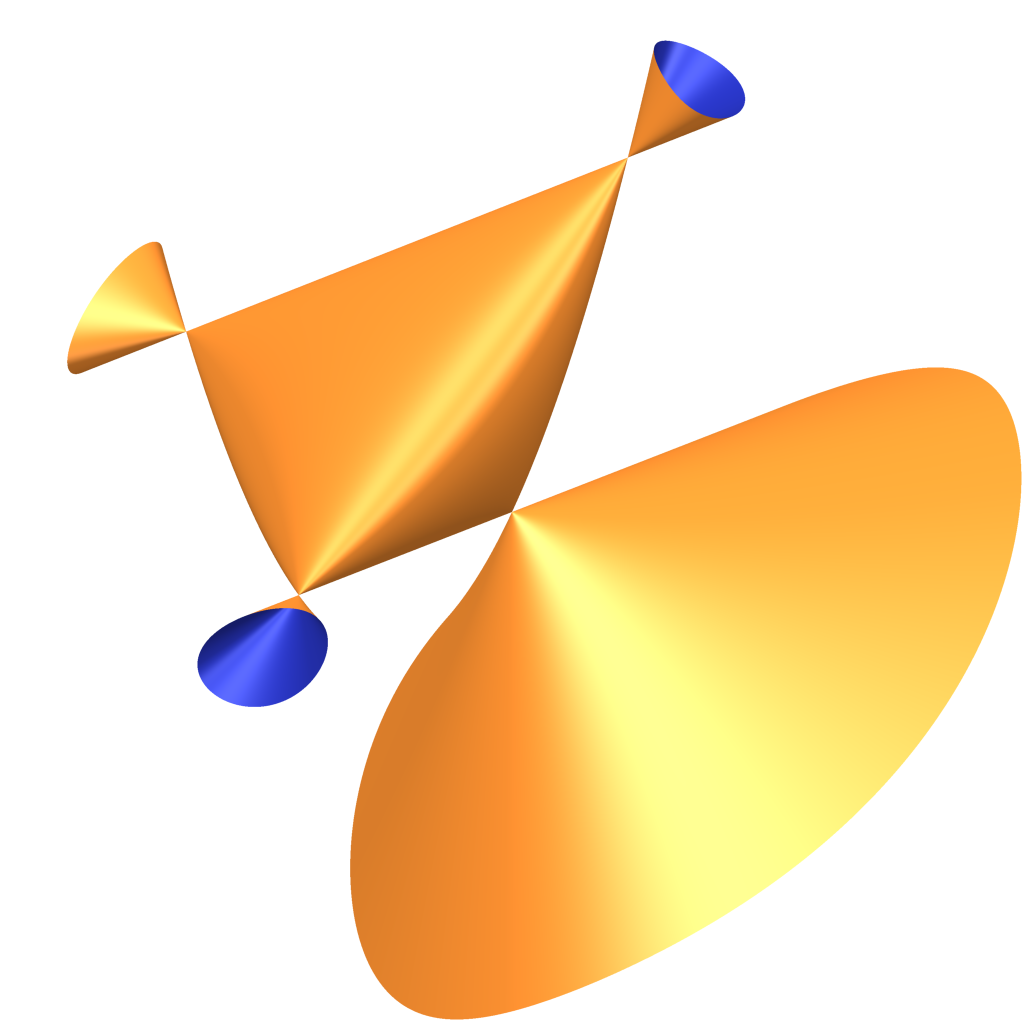}
\includegraphics[width=.32\textwidth]{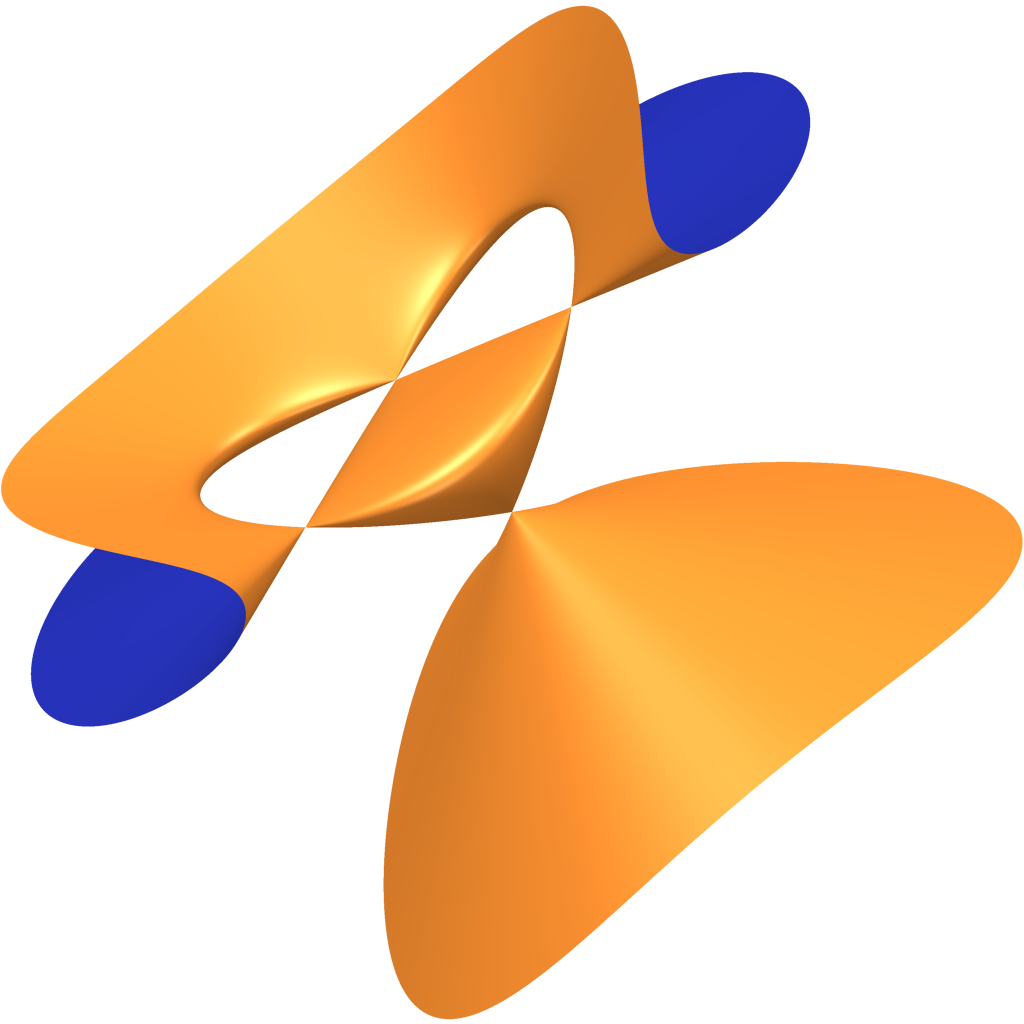}
\includegraphics[width=.32\textwidth]{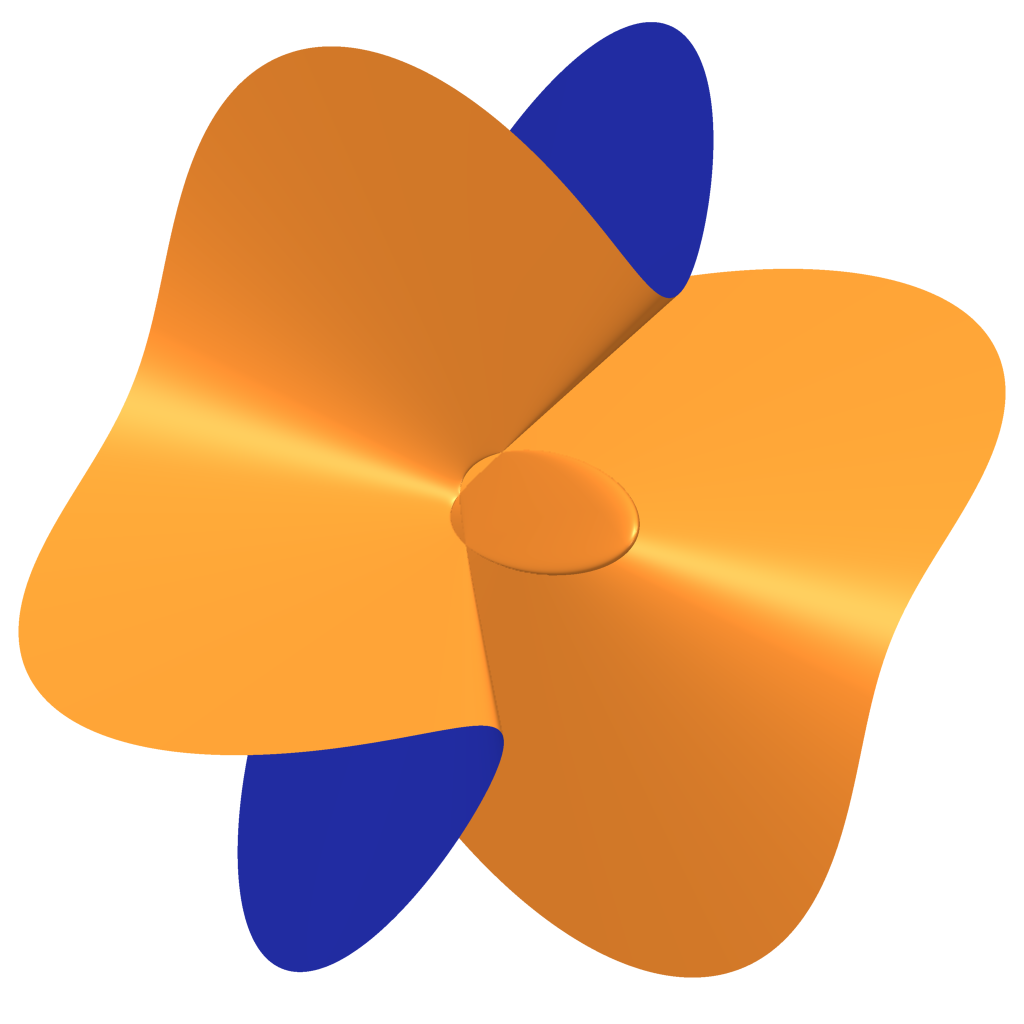}
\caption{Vanishing surfaces of the determinant in Example~\ref{e:lmi}.
In each plot, the bounded region is the spectrahedron.  As the
parameter moves from $\lambda=1$ (left) through $\lambda=0.45$
(middle) to $\lambda=0.2$ (right) it elongates. The ear-shaped cones
emerging from the vertices do not touch in the left-most picture.  As
soon as $\lambda<1$, they do touch: Even in the middle picture, the
cone going off to the bottom and the sheet emerging from the three
remaining vertices are connected in codimension one (outside of the
pictured area).  If $\lambda_1 = 1$, but
$\lambda_2 < 1$, then exactly three of the four vertex cones meet
eventually. \label{f:surferSpectra}}
\end{figure}
Each of the three spectrahedra has four vertices, although this is
hardly visible in the rightmost picture.  These are also the vertices
of~$P$.  In fact, when $\lambda$ is close to 1, the spectrahedron
looks like a bloated version of~$P$.
The analytic center of the LMI is the point $(x,y,z)$ where the
determinant is maximal.  Numerical approximations can be computed
efficiently with semidefinite optimization (we used
\yalmip\cite{YALMIP} in \matlab).  Some values are given in
Table~\ref{tab:values}.
\begin{table}[htp]
\centering
\begin{tabular}{|c|c|}
\hline 
$\lambda$ & analytic center \\
\hline
1 &   $ ( 0.250, 0.250, 0.250 )$  \\
0.8 & $ ( 0.254, 0.254, 0.217 )$  \\
0.5 & $ ( 0.300, 0.300, 0.094 )$  \\
$\sqrt{2} - 1$ & $(0.333, 0.333, 0.000)$ \\
0.4 & $ ( 0.343, 0.343, -0.023 )$  \\
0.2 & $ ( 1.580, 1.580, -2.976 )$ \\
\hline
\end{tabular}
\caption{Coordinates of the analytic center as a function of
$\lambda$.}
\label{tab:values}
\end{table}
Interestingly, the \mosek solver that we used declares the
spectrahedron as unbounded for parameter values $\lambda < 0.171$.
The transition of the $D$-optimal design to a saturated design at
$\sqrt{2}-1$ found in \cite{grasshoff2013optimal} is visible here as
the analytic center leaves the polytope~$P$ at that parameter value.
In this sense, the optimality of certain designs can be understood in
terms of the geometry of deforming spectrahedra.
\end{example}

We close by mentioning another connection between polyhedral and
spectrahedral geometry.  The \emph{elliptope} is the spectrahedron
consisting of all positive semi-definite matrices with entries one on
the diagonal (so-called correlation matrices).  It is a well-known
relaxation of the correlation polytope and its polyhedral faces have
received considerable attention (see
\cite{laurent1995positive,laurent1996facial}).  Example~\ref{e:lmi}
motivates the study of the deformation of the linear matrix
inequalities arising from affine hulls of information polytopes.  Each
such deformation starts at an elliptope when $\beta = 0$.  As $\beta$
becomes more negative, the spectrahedron deforms and eventually its
analytic center leaves the information matrix polytope.  A thorough
understanding of this phenomenon would probably yield new insights
about optimality of experimental designs, in particular
Question~\ref{q:cornerOnly}.

\bibliographystyle{amsplain}
\bibliography{poisson}

\end{document}